\documentclass[12pt]{amsart}

\usepackage{graphicx,psfrag,epsf}
\usepackage{enumerate}

\usepackage{graphicx}
\usepackage{amssymb, amsthm, amsfonts, amsmath}
\usepackage{epstopdf}
\usepackage{bbm}
\usepackage{algorithm}
\usepackage{algpseudocode}

\usepackage{hyperref}
\hypersetup{colorlinks=false,pdfborder={0 0 0}}
\usepackage{color}

\usepackage[margin=1in]{geometry}

\newtheorem{theorem}{Theorem}[section]

\newtheorem{remark}{Remark}[section]
\newtheorem{lemma}{Lemma}[section]

\newtheorem{assumption}{Assumption}

\def\WTGLC{Hard rejection sampling}
\def\return{\textbf{return} }
\def\restart{\textbf{restart}}

\def\X{ \textbf{X} }
\def\Y{ \textbf{Y} }

\def\speedup{\text{\rm speedup}}

\def\XI{\textbf{$X^{(I)}$}}
\def\xI{\textbf{$x^{(I)}$}}
\def\EI{E^{(I)}}
\def\yI{y^{(I)}}

\def\P{{\mathbb P}}

\def\A{\mathcal{A}}
\def\B{\mathcal{B}}
\def\L{\mathcal{L}}
\def\R{\mathbb{R}}
\def\e{\mathbb{E} }

\def\bfa{{\bf a}}
\def\bfw{{\bf w}}

\def\XI{\textbf{$X^{(I)}$}}
\def\xI{\textbf{$x^{(I)}$}}
\def\EI{E^{(I)}}
\def\yI{y_I}

\newcommand{\ignore}[1]{ }

\def\ts{\hskip.03cm}

\begin{document}

\title[Improvements to exact Boltzmann sampling]{Improvements to exact Boltzmann sampling using probabilistic divide-and-conquer and the recursive method}
  \author{Stephen DeSalvo}
\address{UCLA Department of Mathematics, 520 Portola Plaza, Los Angeles, CA, 90095}
\email{stephendesalvo@math.ucla.edu}

\date{\today}

\begin{abstract}
We demonstrate an approach for exact sampling of certain discrete combinatorial distributions, which is a hybrid of exact Boltzmann sampling and the recursive method, using probabilistic divide-and-conquer (PDC). 
The approach specializes to exact Boltzmann sampling in the trivial setting, and specializes to PDC deterministic second half in the first non-trivial application. 
A large class of examples is given for which this method broadly applies, and several examples are worked out explicitly.
\end{abstract}

\ignore{ 

}

\maketitle

\section{Introduction}

The Boltzmann sampler has transformed the way in which combinatorial structures are analyzed and sampled by taking advantage of the generating function structure.  
One starts with a family of combinatorial objects, $\mathcal{C}$, parameterized by various integer-valued statistics like size and number of components, and writes $\mathcal{C}$ as a disjoint union of \emph{finite} sets,
for example,
\[ \mathcal{C} = \bigcup_n \mathcal{C}_n = \bigcup_n \bigcup_k \mathcal{C}_{n,k}. \]
We may have $n$ and $k$ represent, for example, certain statistics like the size of an integer partition and the number of parts, respectively, and $\mathcal{C}_{n,k}$ is the set of all integer partitions of size~$n$ into exactly $k$ parts. 
The goal is then to sample from such a set of objects. 

A standard approach for specifying a sampling algorithm is to name the Boltzmann model, and construct a combinatorial object recursively via the sizes of its components~\cite{Boltzmann};
for example, the part sizes of an integer partition, the block sizes of a set partition, the cycle sizes in a random permutation. 
The \emph{Boltzmann sampler} is then a sampling algorithm which gives a weight to each component-size in proportion to its prevalence in the set of objects of a given size, and does so via a joint distribution of \emph{independent} random variables. 
For unlabelled structures, an object of size~$n$ is generated with probability, for some given real-valued tilting parameter $x$, 
\[\P\left( \mbox{random object is of size~$n$}\right) =  \frac{c_n\,  x^n}{C(x)}, \]
where $c_n$ is the number of objects of size~$n$ and $C(x) = \sum_{n \geq 0} c_n x^n$ is the generating function of the sequence~$c_n$, $n \geq 0$.
For labelled structures, an object of size~$n$ is generated with probability, for some given real-valued tilting parameter $x$, 
\[\P\left( \mbox{random object is of size~$n$}\right) =  \frac{c_n\,  x^n}{n!\, \widehat{C}(x)}, \]
where $\widehat{C}(x) = \sum_{n \geq 0} c_n \frac{x^n}{n!}$ is the exponential generating function  of the sequence~$c_n$, $n \geq 0$.
The result of a Boltzmann sampler is a random object \emph{of random size}; for example, it generates an object in $\mathcal{C}_N$, where $N$ is a random variable with a certain distribution. 

A spectacular property of the Boltzmann sampler is that, conditional on the event $\{N~=~n\}$, the component structure generated is in proportion to the number of objects in $\mathcal{C}_n$ which have that component structure.  
Thus, one immediately obtains an \emph{exact} sampling algorithm for the uniform distribution over $\mathcal{C}_n$ by repeatedly sampling until the event $\{N = n\}$ occurs, discarding samples which do not satisfy this event; this is known as \emph{exact Boltzmann sampling}~\cite{Boltzmann}. 
The limitation of exact Boltzmann sampling is then the probability that a random-sized object generated via a Boltzmann sampler satisfies the event $\{N = n\}$.  
Owing to the plethora of results pertaining to combinatorial enumeration, local limit theorems, and saddle point analysis, one can estimate this probability, and define the \emph{rejection cost} as $\P(N = n)^{-1}$, since it is the \emph{expected} number of times we must sample using the Boltzmann sampler before a sample satisfies the event $\{N=n\}$. 
This rejection cost can grow polynomially or even exponentially in~$n$, depending on the combinatorial structure and the event of interest. 

A general method for the random sampling of combinatorial structures is the recursive method of Nijenhuis and Wilf~\cite{NW, NWBook}. 
The method samples the components of a combinatorial structure one at a time, in proportion to its prevalence in the overall target set, by constructing a table of values based on a recursion that the combinatorial sequences satisfies.
This is equivalent to forming a conditional probability distribution of component-sizes, and is also equivalent to an unranking algorithm, which enumerates all possible objects of size~$n$, say $p(n)$, samples a uniform number between $1$ and $p(n)$, and determines the component structures via the recursion.
Once this table is complete, sampling is efficient. 
The main drawback of this method is that the table size may be overwhelming, and often only a small portion of the table is utilized with high probability, even though the full table is needed in principle. 

Probabilistic divide-and-conquer (PDC) is an exact sampling method which divides a sample space into two separate parts, samples each part separately, and then combines them to form an exact sample from the target distribution; see~\cite{PDC, PDCDSH}. 
It was successfully utilized in~\cite{PDC} to obtain an asymptotically efficient random sampling algorithm for integer partitions. 
A similar approach was used in~\cite{Alonso} for the random sampling of Motzkin words, also obtaining an asymptotically efficient sampling algorithm. 
In both applications, the key to obtaining an asymptotically efficient algorithm was the explicit, efficient computing of certain rejection functions, which are not always present in more general contexts. 
Thus, in our present treatment, we have applied PDC in such a way that the corresponding rejection formulas are always explicit and efficient to compute.  
In addition, our main algorithm, Algorithm~\ref{PDC recursive method}, is embarassingly parallel, see Remark~\ref{remark:parallel}. 

In Section~\ref{sect:other}, we review various available sampling methods. 
In Section~\ref{main:section}, we present our main algorithm, which combines elements of exact Boltzmann sampling with the recursive method. 
An analysis of the costs and benefits are contained in Section~\ref{sect:costs}.  
We apply this idea to integer partitions in Section~\ref{sect:integer:partitions} and to set partitions in Section~\ref{sect:set:partitions}, and demonstrate how this idea generalizes to a larger class of combinatorial structures in Section~\ref{sect:three:classes}. 

\section{Exact sampling}
\label{sect:other}

\subsection{Alternatives to exact sampling}

An alternative to exact Boltzmann sampling is to run a \emph{forward} Markov chain on the state space. 
The main drawback is that, unless we can already sample uniformly from the state space, after any finite number of steps (chosen in advance) there will always be some form of bias in the chain. 
\emph{If} one can prove that the chain is rapidly mixing, then this error is usually considered an acceptable form of bias, as it is often of the same order of magnitude of other forms of errors after a polynomial number of steps. 
However, there are many examples where proving that a Markov chain is rapidly mixing is not so straightforward, see for example~\cite[Chapter 23]{MarkovChainBook}, and other examples where it is proved that mixing takes an exponentially long time, see for example~\cite{NegativeSIS, MCLimitations}.  

Another alternative is the Boltzmann sampler (note the absence of the word \emph{exact}), which samples a random combinatorial structure of \emph{random size} $N$, tilted so that $\e\ts N$ is close to $n$, with each object of a given size equally likely. 
There are many quantitative reasons why accepting a random sample of a random size serves as a good surrogate for an exact sample. 
A primary example is the limit shape of integer partitions, see~\cite{PittelShape}, where it was shown that the limit shape of integer partitions coincides with the limit shape obtained by a Boltzmann model; see also~\cite{Bogachev, LD, VershikKerov, Yakubovich} for related results. 
However, it is shown in~\cite{PittelSetPartitions} for set partitions that there exist statistics which are qualitatively different, even asymptotically as $n$ tends to infinity, depending on whether or not the true joint distribution of component-sizes is used; i.e., whether or not the event $\{N = n\}$ is required for all samples. 

A standard approach to improve on exact Boltzmann sampling is to consider an event of the form $E_{n,\epsilon} = \{N \in (n(1-\epsilon),n(1+\epsilon))\}$, for some $\epsilon>0$.  
This effectively widens the target by a small factor of $n$, and often improves the rejection rate to $O(1)$. 
To see this, we note that, as in~\cite{IPARCS}, many Boltzmann samplers with appropriately chosen tilting parameter $x$ produce random target sizes $N$ which are asymptotically normally distributed with mean $n$ and standard deviation $O(n^{-a})$, for some $a > 0$.  
This means, then, that the exact Boltzmann sampler rejects an expected $O(n^{a})$ number of samples before a sample is accepted. 
When $a < 1$, e.g., $a=3/4$ in the case of integer partitions, this implies that eventually, for large enough $n$, all approximate samples will be accepted, making this approach asymptotically equivalent to a Boltzmann sampler. 

\subsection{Other exact sampling methods}

An alternative to running a standard Markov chain forward in time is \emph{Markov chain coupling from the past}, see~\cite{ProppWilson}, where one instead runs simultaneously a Markov chain \emph{on every state in the state space}, starting from some time in the past, forward in time, and couples together chains when they transition into the same state, treating them as the same chain from that time forward. 
If after one step forward in time, starting from time $-1$, all chains are not coupled, then we restart from time $-2$ and run the chain forward two steps, coupling the chains as they coincide.  If not all chains are coupled at time $0$, we reset at time $-2^b$, for $b = 1, 2, \ldots$, until all chains are coupled at time $0$, at which point the chain is in \emph{exact} stationarity.  
This approach has obvious drawbacks, but can also be very effective when there exists a monotonic structure on the transitions and a coupling which allows us to only consider a few extreme chains, with the implication that all chains will be coupled once those extreme chains are coupled. 
Once all chains are coupled, we do indeed have an \emph{exact} sample in finite time.  See~\cite{huber2015perfect} for further examples. 

Another approach intimately related to exact sampling is importance sampling, where instead of demanding an exact sample from a structure, one instead demands the ability to associate a weight to the generated sample, which is a measure for the bias in the sampling algorithm.  The weights can then be used to obtain unbiased estimates of statistics. 
An importance sampling algorithm can be converted into an exact sampling algorithm by applying a rejection to the generated sample. 
The rejection may be particularly severe, as e.g., one very special case of contingency tables~\cite{NegativeSIS}, where it was shown that the weights can be exponentially small.  

We should also note, as is often the case with fundamental combinatorial structures, that alternative sampling algorithms exist which are tailored to the specific form of the components and their intricate dependencies.  For example, one would not attempt to compete with the Fisher-Yates shuffle~\cite{FisherYates} to generate a random permutation, nor is it likely to be fruitful to generate a random set partition according to the Ewen's measure in block structure form more optimally than the Chinese restaurant process, see for example~\cite{Aldous1985}. 
However, if one deviates from the classical form of the combinatorial structure, then it is not always apparent how to adapt these sampling algorithms. 

\subsection{The Recursive Method}
\label{sect:recursive}

The recursive method \cite{NW, NWBook} exploits the recursive nature of a combinatorial sequence in order to extract the conditional distribution of component-sizes in a random sample. For example, letting $p(n,k)$ denote the number of integer partitions of size~$n$ into parts of size at most $k$, we have the well-known recursion
\begin{equation}\label{pnk:recursion} p(n,k) = p(n-k,k) + p(n,k-1) \qquad 1 \leq k \leq n,  \end{equation}
with $p(k,0) = 1$ when $k \geq 0$, $p(k,n) = p(n,n)$ when $k > n$, and $p(k,n) = 0$ otherwise. 
This recursion encodes the idea that we can build a partition of size~$n$ into parts of size at most $k$ by either appending another part of size~$k$ and repeating with $n$ replaced by $n-k$, or by deciding that there shall be no more parts of size~$k$, and continuing with $k$ replaced by $k-1$. 

To obtain a uniform measure, we simply weight these decisions appropriately, and note a surprising independence of our decision at each step; i.e., once we make a decision, the sampling problem restarts with smaller parameters, and is then independent of previous decisions, depending only on the current input parameters $k$ and $n$. 
In this way, it is straightforward to sample part sizes one at a time using a single table of size $n \times n$ until we reach a trivial completion.  

The main drawback is the requirement that we are able to compute the values of $p(n,k)$ exactly, or at least with enough precision on demand to decide definitely between the two courses of action; see Remark~\ref{ADZ} below. 
The dimension of the table is a priori $n \times n$, and with the asymptotic analysis in~\cite{ErdosLehner}, specifically in the special case of integer partitions, only the entries in the first $t\ts \sqrt{n}\, \log(n)$ rows are needed with high probability, taking $t>0$ large enough.  

An alternative recursive approach, the one originally developed in~\cite{NW}, is to consider a recursion on the sequence $p(n)$, the number of integer partitions of $n$, directly, and use its combinatorial interpretation to sample the combinatorial structure in a similar albeit inherently distinct manner.  
A well-known recursion due to Euler is 
\begin{equation}\label{Euler:recursion} n\, p(n) = \sum_{m < n} \sigma(n-m)\, p(m), \qquad n \geq 1, \end{equation}
where we take $p(0) = 1$, and $\sigma(m)$ is the sum of all divisors of $m$. 
This recursion can be seen by writing out $n$ copies of all $p(n)$ partitions of $n$, and then combining partitions of $m$ with certain partitions of $(n-m)$ via divisors of $n-m$.  

To obtain a uniform distribution, one samples a random variable $X$ with distribution given by 
\[ \P(X = m) = \frac{\sigma(n-m)\, p(m)}{n\, p(n)}, \qquad m=0,1,\ldots,n-1. \]
This random variable captures the correct proportion of partially completed partitions, after which we must determine the part sizes $d$ in proportion to the number of partitions of $m$ corresponding to divisors of $n-m$, i.e., 
\[ \P(Y = d) = \frac{d}{\sigma(n-m)}, \qquad  d\, |\, (n-m). \]
Once we have chosen this $d$, we then fill in $(n-m)/d$ parts of size $d$ and update $n$ to be the value~$m$ and repeat. 

\begin{remark}\label{ADZ}
{\rm In order to extend from floating-point accuracy to arbitrary accuracy, it has been pointed out by many authors, see for example~\cite[Section~4]{denise1999uniform} and~\cite[Section~5.2]{PDC}, that one does not need to compute all quantities in an exact sampling procedure to arbitrary precision initially, as long as one can keep track of sufficiently small intervals for which the exact quantities lie, and further precision is available on demand. 
This applies to both numerical calculations as well as generation of random variables, and is referred to as the ADZ method (after Alonso, Denise, Zimmerman) in~\cite{denise1999uniform}.
}\end{remark}

Many straightforward generalizations to~\eqref{pnk:recursion} and~\eqref{Euler:recursion} have previously been exploited for integer partitions, see for example~\cite{ErdosElementary}. 
A very broad generalization of~\eqref{pnk:recursion}, applicable to more than just integer partitions, is contained in~\cite[Chapter 13]{NWBook}, where it is noted that many combinatorial sequences satisfy a recurrence relation of the form
\[ a(n,k) = \varphi(n,k) a( n_w, k) + \psi(n,k) a(n_s,k-1), \]
where $\varphi, \psi$ are given explicitly depending on the combinatorial family, and $n_w$ and $n_s$ are typically of the form $n-a$ for some $a \geq 0$.  

A generalization to~\eqref{Euler:recursion} is also included in~\cite[Postscript:~deux ex machina]{NWBook}, which is connected to the ``prefab" concept of~\cite{BenderGoldman}. 
As is often the case, it is easiest to think of these generalizations as originating from a special case like integer partitions. 
Briefly, one attempts to decompose a combinatorial structure into ``prime" components with multiplicities, which is then used to obtain the form of the generating function and establish recurrence relations. 
For integer partitions, the prime components are the positive integers, and every integer partition of $n$ can be uniquely decomposed into components of sizes $1, 2, \ldots, n$ with multiplicities. 
There are of course technical conditions which must be satisfied, but under reasonable assumptions on how to synthesize two combinatorial objects the idea generalizes to other ``decomposable" combinatorial structures in a natural manner; see~\cite{NWBook} for more examples.

\subsection{Probabilistic Divide-and-Conquer}
\label{sect:pdc}

Probabilistic divide-and-conquer (PDC) is a technique for \emph{exact} sampling, which divides a sample space into two pieces, samples each separately, and then combines them to form an \emph{exact} sample from the target space. 
In this paper, our focus is on a particular parameterization of a sample space specifically suited to Boltzmann sampling. 
Rather than recursively build a Boltzmann model, we instead assume a target sample space which can be written as a joint distribution of real-valued random variables as follows: for each integer $n\geq 1$, let $\X = (X_1, X_2, \ldots, X_n)$ denote an $\R^n$--valued joint distribution of mutually independent random variables, and denote by $\mathcal{F}$ the Borel $\sigma$-algebra of measurable events on $\R^n$.  
Given a set $E_n \in \mathcal{F}$, we define the distribution of $\X_n'$ as
\begin{equation}
\label{distribution}
\L(\X_n') := \L\left((X_1, X_2, \ldots, X_n)\ \Big|\ \X \in E_n \right).
\end{equation}
Many exact Boltzmann samplers, in particular the examples in~\cite{Boltzmann}, can be described in this context using the event $E_n~=~\{\sum_{i=1}^n i\, X_i = n\},$ where the weighted sum is attributing weight~$i$ to component~$i$.  

A ubiquitous first technique for sampling from conditional distributions of the form~\eqref{distribution} is rejection sampling~\cite{Rejection}, for which we describe two main forms.  The first is to sample from the unconstrained distribution $\L(\X)$ and reject with probability $1$ if the event $\{\X \in E_n\}$ is not satisfied; we refer to this form of rejection sampling as \emph{hard rejection sampling} since the rejection probability is in the set $\{0,1\}$.  The second form samples from some alternative distribution $\L(\Y)$, for which the rejection probability lies in the interval $[0,1]$, and is rejected depending on the observed outcome of the sample, say $a$, with some auxiliary randomness; we refer to this form as \emph{soft rejection sampling}, since it requires an auxiliary random variable $U$, uniform over the interval $[0,1]$, and the computation of a function $t(a)$, with the decision to reject only when the event $\{U > t(a)\}$ occurs. 

For our particular parameterization, the hard rejection sampling algorithm is to sample from $\L\left(X_1, X_2, \ldots, X_n\right)$ repeatedly until the event $\{ \X \in E_n\}$ occurs, which is equivalent to an exact Boltzmann sampler.  
The overall number of rejections is geometrically distributed, see for example~\cite{devroye}, with expected value $\P\left( \X \in E_n\right)^{-1}$.

PDC allows us to fashion divisions which attempt to lower the total amount of uncertainty at any given stage of the algorithm, and hence improve upon the rejection cost. 
To apply PDC, we choose a division of the sample space consisting of $A \in \mathcal{A}$ and $B \in \mathcal{B}$, where $A$ and $B$ are independent and can be sampled separately, and the target set $S \in \mathcal{A} \times \mathcal{B}$ can be described as \[\{ (A,B) \in \mathcal{A}\times \mathcal{B} : (A,B) \in E\},\] where $E$ is an event either of positive probability or which satisfies a regularity condition. 
The PDC Lemma below motivates an approach for exact sampling.

\begin{lemma}[PDC Lemma~\cite{PDC}]
\label{PDC:lemma}
Assume $E$ is an event of positive probability.  
Suppose $X$ is a random element of $\A$ with distribution 
\begin{equation}\label{def X}
   \L(X) = \L( \, A \, | \, (A,B)\in E \, ),
\end{equation}
and  $Y$ is a random element of $\B$ with conditional distribution
\begin{equation}\label{def Y}
   \L(Y \, | X=a ) = \L( \, B \, | \, (a,B)\in E \, ).
\end{equation}
Then $\L(X,Y) = \L( (A,B) | (A,B) \in E)$.
\end{lemma}

Algorithms~\ref{WTGL procedure} and~\ref{PDC procedure} below present the standard hard rejection sampling algorithm and the standard PDC sampling algorithm, respectively, in the language of a PDC division. 
Note that the designer of the algorithm must specify the division in advance, and that PDC algorithms are very sensitive to the specified division, since one must be able to sample from the corresponding conditional probability distributions.

\begin{algorithm}[H]{\rm
\begin{algorithmic}
\State 1.  Generate sample from $\L(A)$, call it $a$.
\State 2.  Generate sample from $\L(B)$, call it $b$.
\State 3.  Check if $(A,B) \in E$; if so, return $(a,b)$, otherwise restart.
\end{algorithmic}}
\caption{\WTGLC\ from $\L( (A,B)\, |\, (A,B)\in E)$}
\label{WTGL procedure}
\end{algorithm}

\begin{algorithm}[H]{\rm
\begin{algorithmic}
\State 1. Generate sample from $\L(A\, |\, (A,B) \in E)$, call it $x$.
\State 2. Generate sample from $\L(B\, |\, (x,B) \in E)$ call it $y$.
\State 3. Return $(x,y)$.
\end{algorithmic}}
\caption{Probabilistic Divide-and-Conquer sampling from $\L( (A,B)\, |\, (A,B)\in E)$} 
\label{PDC procedure}
\end{algorithm}

As a first approach for fashioning an explicit and practical PDC algorithm, we modify Algorithm~\ref{PDC procedure} above to utilize soft rejection sampling for the sampling of the first conditional distribution $\L(A\, |\, (A,B) \in E)$, and present this algorithm in Algorithm~\ref{PDC procedure von Neumann} below. 

\begin{algorithm}[H]{\rm
\begin{algorithmic}
\State 1. Generate sample from $\L(A),$ call it $a$.
\State 2. Accept $a$ with probability $t(a)$, where $t(a)$ is a function of $\L(B)$ \\
\ \ \ \ and $E$; otherwise, restart.
\State 3. Generate sample from $\L(B\, |\, (a,B) \in E),$ call it $y$.
\State 4. Return $(a,y)$.
\end{algorithmic}}
\caption{Probabilistic Divide-and-Conquer sampling from $\L( (A,B)\, |\, (A,B)\in E)$ using soft rejection sampling}
\label{PDC procedure von Neumann}
\end{algorithm}

At this point, it is apparent that two quantities are necessary to apply this PDC algorithm
\begin{enumerate}
\item The rejection function $t(a)$, for each $a \in \mathcal{A}$;
\item $\L(B\, |\, (a,B) \in E)$ for each $a\in \mathcal{A}$.
\end{enumerate}

Given our assumed parameterization of the target sample space, for many reasonable choices of divisions it is often straightforward to write down an explicit expression for $t(a)$, which we shall demonstrate shortly. 
It is not necessarily straightforward to evaluate $t(a)$, however, which is why previous PDC algorithms have utilized divisions which make $t(a)$ explicit and efficient to compute; see~\cite{PDC, PDCDSH}. 
In addition, previous PDC algorithms have either been fashioned such that $|\{(a,B)\in E\}| = 1$ for each $a \in \mathcal{A}$, i.e., deterministic second half~\cite{PDCDSH}; or, where $\L(B\, |\, (a,B)\in E)$ is equivalent to a reduced version of $\L(A\, |\, (A, B)\in E)$, i.e., self-similar PDC; see~\cite[Section~3.5]{PDC}, see also Section~\ref{sect:selfsimilar}.  

\subsection{PDC deterministic second half}
\label{sect:pdcdsh}

In~\cite{PDCDSH}, a general framework is presented for random sampling using Algorithm~\ref{PDC procedure von Neumann} when $|\{(a,B)\in E\}| = 1$. 
The main algorithm in the discrete setting is Algorithm~\ref{PDC discrete} below, which also serves as an important special case to our main algorithm, Algorithm~\ref{PDC recursive method}, in Section~\ref{main:section}. 
We first introduce some notation. 

We shall always assume that $n$ is a large, finite positive integer. 
The set $I = \{i_1, i_2, \ldots\} \subset \{1,2,\ldots, n\}$ will denote some fixed, finite index set of positive integers. 
Given such an index set $I$, we define $\mathcal{A} \equiv \mathcal{A}_I = \R^{|I|}$, $\mathcal{B} \equiv \mathcal{B}_I = \R^{n-|I|}$, with 
 \[ \XI = (X_i)_{i \notin I} \in \mathcal{A}, \qquad \qquad X_I = (X_i)_{i \in I} \in \mathcal{B},\]
 and 
\[\EI := \{x \in \mathcal{A} : \exists y\in \mathcal{B} \mbox{ such that } (x,y)\in E\}.\]
We also define the function $\sigma_I : \R^{n-|I|} \times \R^{|I|}$ to be the operation which combines the elements in two vectors, say $x = (x_1,\ldots, x_{n-|I|})$ and $y = (y_1,\ldots, y_{|I|})$ in such a way that $z = \sigma_I( (x_1, \ldots, x_{n-|I|}), (y_1, \ldots, y_{|I|}) )$ is the (unique) permutation of size~$n$ such that the elements of $x$ and $y$ maintain their original order, with elements $z_{i_j} = y_j$, for $j=1,\ldots, |I|$. 
In other words, we wish to divide up the sample space $(X_1, \ldots, X_n)$ via the set~$I$, which will vary by example, work with $\XI = (X_i)_{i \notin I}$ and $X_I = (X_i)_{i \in I}$ separately, and then denote, e.g., the acceptance event as $\{ \sigma_I(\XI, X_I) \in E\}$. 

We shall also let $U$ denote a uniform random variable in the interval $(0,1)$, independent of all other random variables, and $u$ will denote a random variate generated from this distribution. 

\begin{algorithm}
\caption{\cite{PDCDSH} PDC deterministic second half for independent discrete random variables} 
\begin{algorithmic}
\Procedure {Discrete\_PDC\_DSH}{$X_1, X_2, \ldots, X_n, E, I$} \\
\State {\bf Assume:} $I = \{i\}$, for some $1 \leq i \leq n$.
\State {\bf Assume:} For each $\xI \in \EI$, there is a unique $\yI$ such that $\sigma_I(\xI,\yI) \in E$. \\ 
\State Let $\XI := (X_1, \ldots, X_{i-1}, X_{i+1},\ldots X_n)$. 
\State Sample from $\L(\XI)$, denote the observation by $\xI$.
\State Let $\yI$ denote the unique value such that $\sigma_I(\xI, \yI)\in E$.  
\If {$\xI \in \EI$ \text{ and } $u < \frac{P\left(X_i = \yI\right)}{\max_\ell P(X_i = \ell)}$ }
		\State \return $\sigma_I(\xI,\yI)$
\Else
\State \restart
\EndIf
\EndProcedure
\end{algorithmic} \label{PDC discrete}
\end{algorithm}

It is perhaps surprising that such a simple division, i.e., using $I = \{i\}$ for some $1 \leq i \leq n$ so that 
\[ \XI = (X_1, \ldots, X_{i-1}, X_{i+1}, \ldots, X_n) \qquad \mbox{and}\qquad X_I = (X_i),\]
 produces an \emph{automatic speedup} over hard rejection sampling, in terms of the expected number of rejections, at the cost of evaluating the probability mass function of $X_i$ and computing its maximum value. 
To see that this is indeed more efficient, consider the acceptance event for rejection sampling from $\X \in E$.  
Given any $\xI \in \EI$, let $\yI \equiv \yI(\xI)$ denote the unique value such that $\sigma_I(\xI, \yI)\in E$.  
   The acceptance event for hard rejection sampling can be written as 
\begin{equation} \label{WTGL discrete}\{ \text{$\XI \in \EI$ and $U < P(X_I = \yI)$}\}. \end{equation}
The acceptance event for Algorithm~\ref{PDC discrete} can be written as 
\begin{equation}\label{PDCTSH discrete}
\left\{ \text{$\XI \in \EI$ and $U < \frac{\P( X_I = \yI)}{\max_\ell \P(X_I = \ell)}$}\right\}.
\end{equation}
The added efficiency comes from accepting the sample \emph{in proportion} to the likelihood of the remaining uncertainty in~\eqref{PDCTSH discrete}, rather than its likelihood in~\eqref{WTGL discrete}. 
This approach, then, favors selecting indices $I$ for which the distribution of $X_I$ is not dominated by a single point mass, i.e., a small maximum point probability.  
There is a similar adaptation for continuous random variables, see~\cite{PDCDSH}, where in many situations of interest the default rejection sampling algorithm has an infinite expected wait time, and the analogous PDC deterministic second half algorithm has a finite expected wait time. 

\begin{remark}\label{default:remark}
{\rm All exact Boltzmann samplers which can be written in terms of~\eqref{distribution}, with components consisting of explicitly computable probability mass functions, can take advantage of Algorithm~\ref{PDC discrete}, as any selection of index $I$ is guaranteed to reduce the expected number of rejections, at the cost of what is often a simple and explicit arithmetic calculation.
}\end{remark}

\subsection{Self-similar PDC}
\label{sect:selfsimilar}
The PDC deterministic second half approach of the previous section, while offering a simple, guaranteed speedup in many cases of interest, can be improved if more knowledge of the distributions and conditioning event $E$ is available. 
As was noted earlier in Algorithm~\ref{PDC procedure von Neumann}, it is possible to sample from $\L(A\, |\, (A,B)\in E)$ by sampling from $\L(A)$ and applying an appropriate rejection.  
In many cases of interest, the division is such that the remaining part, $\L(B\, |\, (a,B) \in E)$ is equivalent to the original sampling problem with smaller values of parameters.  

Such an approach was utilized in~\cite{Alonso} for the exact random sampling of Motzkin words, yielding an overall asymptotically constant rejection rate. 
The general principle for the aforementioned application was developed independently in~\cite{PDC} and given the name self-similar PDC, and was used to produce an exact sampling algorithm for integer partitions with an overall asymptotic rejection rate of at most $2\sqrt{2}$. 

The main cost associated with this approach is the calculation of the rejection function $t(a)$, which was available for Motzkin words as the quotient of binomial coefficients, and for integer partitions due to the enumeration results of Hardy and Ramanujan~\cite{HR}, Rademacher~\cite{Rademacher}, and Lehmer~\cite{Lehmer}. 
In many other cases, especially when leaving the realm of fundamental combinatorial structures, such enumeration formulas are often not available. 

\section{PDC and the recursive method}
\label{main:section}

We now present the main algorithm for exact sampling via PDC and the recursive method.  
The combination of the recursive method and PDC is designed to control the size of the table required for the recursive method, while at the same time improve on the rejection probability of exact Boltzmann sampling and PDC deterministic second half, without requiring any complicated auxiliary calculations of rejection functions as discussed in Section~\ref{sect:selfsimilar}.
Also, for this section recall the notation that $U$ denotes a random variable with the uniform distribution over the unit interval $[0,1]$ and $u$ denotes a random variate generated from this distribution. 

To demonstrate the method, let us start by extending the PDC deterministic second half algorithm of Section~\ref{sect:pdcdsh}, so that two components are sampled in the second stage; i.e., let $I = \{1,2\}$, so that 
\[ \XI = (X_3, \ldots, X_n) \qquad \mbox{and} \qquad X_I = (X_1, X_2), \]
and we take $E = \{ \sum_{i=1}^n i\, Z_i = n\}$. 
Then $\EI = \{ \sum_{i=3}^n i\, Z_i \leq n\}$, and 
the acceptance event for the first stage of Algorithm~\ref{PDC procedure von Neumann} can be described by 
\begin{equation}
\label{small:accept}
\left\{ \text{$\XI \in \EI$ and $U < \frac{\P( X_1 + 2X_2 = y_I)}{\max_\ell \P(X_1 + 2X_2 = \ell)}$}\right\}.
\end{equation}
(In this setting, even though $y_I$ is uniquely determined, $X_I$ may not be, since the set $\{(x_1,x_2) : x_1 + 2x_2 = y_I\}$ may consist of a great many elements.)
Once this outcome is accepted, we then have the task of sampling from 
\begin{equation}\label{small:sample}\L\left( (X_1, X_2)\, \middle|\, X_1 + 2X_2 = \yI\right).\end{equation}
Note that this is a reduced problem, but not identical to the original problem, since $\yI$  is not guaranteed to be $2$. 

At this point, we pause to note that this approach requires two further tasks: \\
\begin{enumerate}
	\item Calculation of $\frac{\P( X_1 + 2X_2 = y_I)}{\max_\ell \P(X_1 + 2X_2 = \ell)}$ in~\eqref{small:accept}. \\
	\item Sampling from $\L\left( (X_1, X_2)\, \middle|\, X_1 + 2X_2 = \yI \right)$ in~\eqref{small:sample}.  \\
\end{enumerate}

In this small case, the two tasks above can often be handled by brute force and/or ad hoc methods; however, at this point we make a simplifying assumption on the joint distribution $(X_1, \ldots, X_n)$, one which is not necessary to apply PDC in general, but which is often satisfied in examples involving Boltzmann sampling and makes the utilization of the recursive method practical.

\begin{assumption}[Boltzmann Assumption]\label{Boltzmann:assumption}
Assume for each $I \subset \{1,\ldots, n\}$ and $\ell\geq 0$, the joint distribution $X_I$ is such that we have 
\begin{equation}\label{rI:equation} r_I(\ell) := \P(X_{i_1} = z_{i_1},  X_{i_2} = z_{i_2}, \ldots) \end{equation}
for any collection of constants $z_{i_1}, z_{i_2}, \ldots$ satisfying $\sum_{i\in I} i\, z_i = \ell$; i.e., $r_I(\ell)$ does not depend on the $z_{i_1}, z_{i_2}, \ldots$, only its weighted sum.  
Then, letting $a_I(\ell)$ denote the number of such collections, we may write  
\begin{equation}\label{equation:r0}
 \P\left( \sum_{i \in I} i\, X_i = \ell\right) = \sum_{z_{i_1},z_{i_2},\ldots : \sum_{i\in I} i\, z_i = \ell}\P(X_{i_1} = z_{i_1},  X_{i_2} = z_{i_2}, \ldots) = a_I(\ell)\, r_I(\ell).
 \end{equation}
In addition, we assume that the sequence $a_I(\ell)$, $\ell \geq 0$, satisfies a recursion which is amenable to applying the recursive method. 
\end{assumption}

\begin{remark}
{\rm Many Boltzmann samplers utilize tilting parameters, say $x$ and $\theta$, whose value does not affect the unbiased nature of the algorithm, and whose purpose is to optimize the probability that the target is hit.  
In terms of Assumption~\ref{Boltzmann:assumption}, this means that the righthand side of~\eqref{equation:r0} can be written as 
\[ a_I(\ell)\, r_I(\ell, x, \theta), \]
and the key property remains, which is that the probability of generating an object of a given weight depends only on the weight, and not on the particular component structure.  
Another particularly advantageous aspect of PDC is that once the first stage is performed, i.e., \emph{after} we have applied the rejection step and locked in the observation for $\XI$, we may subsequently adjust the tilting parameters for the second stage, choosing their values to optimize the completion of the remaining sampling algorithm; see~\cite[Section~4.3.1]{PDC}. }
\end{remark}

All of our examples will henceforth be assumed to satisfy Assumption~\ref{Boltzmann:assumption}, even if not explicitly stated. 
Generalizing this approach, for any $k \geq 1$ we next consider divisions of the form 
\[ \XI = (X_{k+1}, \ldots, X_n) \qquad \mbox{and} \qquad X_I = (X_1, \ldots, X_k), \]
and the acceptance event is given by, with $I = \{1, \ldots, k\}$, 
\begin{equation}
\label{two:accept}
\left\{ \text{$\XI \in \EI$ and $U < \frac{\P( \sum_{i=1}^k i\, X_i = \yI)}{\max_\ell \P(\sum_{i=1}^k i\, X_i = \ell)}$}\right\}.
\end{equation}
As stated previously, the main impediment for applying PDC to a chosen division is calculating the rejection probability, and sampling from the remaining conditional distribution,  and the recursive method solves both tasks!
To see this, let us rewrite~\eqref{two:accept} using Assumption~\ref{Boltzmann:assumption}:
\begin{equation}
\label{two:accept:2}
\left\{ \text{$\XI \in \EI$ and $U < \frac{a_I(\yI)\,r_I(\yI)}{\max_\ell a_I(\ell)\,r_I(\ell)}$}\right\}.
\end{equation}
Thus, in order to evaluate the acceptance event, we need to know the values of $a_I(\ell)$ for $\ell = 0,1,\ldots$, which can be obtained via a recursion on the sequence $a_I(\ell)$, as well as the values of $r_I(\ell)$, $\ell=0,1,\ldots$, which will be obtained from calculations derived from the particular combinatorial structure; see sections~\ref{sect:integer:partitions},~\ref{sect:set:partitions}, and~\ref{sect:three:classes} for explicitly worked out examples. 

The second part, i.e., sampling from $\L\left( X_I\, \middle|\, \sum_{i \in I} i\, X_i = \yI \right)$, is in fact precisely the distribution that the recursive method samples from using a table, we need only supply an appropriate recursion for the sequence $a_I(\ell)$, $\ell \geq 0$. 
Algorithm~\ref{PDC recursive method} below describes the procedure assuming one is able to create and randomly access such a table of values from the recursive method. 

\begin{algorithm}[H]
\caption{PDC with the recursive method} 
\begin{algorithmic}
\Procedure {PDC\_with\_Recursive\_Method}{$X_1, X_2, \ldots, X_n, E, I$}  \\
\State {\bf Assume:} $(X_1, \ldots, X_n)$ satisfies Assumption~\ref{Boltzmann:assumption}. 
\State {\bf Assume:} $E = \left\{\sum_{i=1}^n i\, X_i = n\right\}.$ \\

\State 1: Generate table $T$ via the recursive method, with $T(j) = a_I(j)$ for $j \geq 0$. 
\State\label{line:one} 2: Sample from $\L(\XI)$, denote the observation by $\xI$, and let $m = \sum_{i \notin I} i\, x_i$. 
\State 3: Let $\yI  = n-m$ and 
\[ t(\yI) = \frac{T(\yI)\, r_I(\yI)}{\max_\ell T(\ell)\, r_I(\ell)}. \]
\State 4: {\bf If} {$\xI \in \EI$ \text{ and } $u < t(\yI)$ } {\bf then}
	\State 5: \hskip .17in Generate $x_I$ from $\L(X_I\, |\, \sigma_I(\xI,X_I)\in E)$ via the recursive method. 
	\State 6: \hskip .17in Return $\sigma_I(x_I, \xI)$.
\State 7: {\bf else} 
\State 8: \hskip .17in Repeat
\State 9: {\bf endif} 
\EndProcedure
\end{algorithmic} \label{PDC recursive method}
\end{algorithm}

\begin{theorem}
Algorithm~\ref{PDC recursive method} generates an unbiased sample from the distribution~\eqref{distribution}.
\end{theorem}
\begin{proof}
The rejection function $t(\yI)$ in Line~3 is defined such that once the algorithm reaches Line~5, the sample $\xI$ has distribution $\L(A\, |\, \sigma_I(A,B) \in E)$ for $A = \XI,$ $B = X_I$ and $E = \left\{\sum_{i=1}^n i\, X_i = n\right\}.$
The recursive method generates the remaining part of the sample according to the conditional distribution~$\L(B\, |\, \sigma_I(\xI,B) \in E)$. 
By Lemma~\ref{PDC:lemma}, $(x_I, \xI)$ is an exact sample from $\L((A,B)\, |\, \sigma_I(A,B) \in E)$. 
\end{proof}

\begin{remark}\label{remark:parallel}
{\rm An advantage of Algorithm~\ref{PDC recursive method} is that the generation of the table in Line~1 and the first stage of sampling in Line~2 can be performed concurrently, which is ideal when a large number of samples are desired. 
That is, while we are generating the table, we may generate concurrently the samples $\xI_1, \xI_2, \ldots$. 
Then, once the table is complete, applying rejection and sampling the remaining parts $x_I$ from the table is efficient. }
\end{remark}

\section{Cost of the algorithm}
\label{sect:costs}
The overall cost of the algorithm consists of 
\begin{enumerate}
\item the cost to sample from $\L(\XI)$, and the expected number of rejections before acceptance;
\item the cost to generate and store the table; 
\item the cost to generate a random object via the table.
\end{enumerate}

The cost to sample $\L(\XI)$ is one which we shall not analyze in great detail, except to point out that the na\"ive sampling of each coordinate of $\XI$ separately may not be optimal; see for example the discussions in sections~\ref{sect:integer:partitions} and \ref{sect:set:partitions}. 
Fortunately, the cost to sample $\L(\XI)$ is unrelated to the rejection cost, which is our main metric for algorithmic efficiency. 

The expected number of rejections in rejection sampling is given by (see~\cite{Rejection})
\[ \frac{\max_\ell {\P( (X_1, \ldots, X_n) \in E | \XI = \ell)}}{\P((X_1, \ldots, X_n) \in E)}. \]
That is, it is the quotient of the overall probability of landing in the target, with an added boost from soft rejection sampling. 
We define the \emph{boost factor} of a PDC algorithm as the inverse of this maximal probability, i.e., 
\[ \mbox{boost factor} = \frac{1}{\max_\ell {\P( (X_1, \ldots, X_n) \in E | \XI = \ell)}}. \]
To summarize, whereas the expected number of rejections in exact Boltzmann sampling is 
\[ \frac{1}{\P((X_1, \ldots, X_n) \in E)}, \]
using PDC we obtain an expected number of rejections which is
\[ \frac{\max_\ell {\P( (X_1, \ldots, X_n) \in E | \XI = \ell)}}{\P((X_1, \ldots, X_n) \in E)}. \]

The arithmetic cost to generate a table via the recursive method, and to generate a random object from that table, has been studied previously, see for example~\cite{denise1999uniform} and the references therein, and so we refer the interested reader to their treatment. 

Finally, we note that the degree to which the combination of PDC and the recursive method is an improvement overall depends on the choice of the PDC division, which we now highlight with specific examples. 

We next introduce several standard definitions regarding the order of growth of a function. 
For two real-valued functions $f$ and $g$ and $n$ a real number, 
we say $f(n) = O(g(n))$ if and only if there are constants $C>0$ and $n_0 > 0$ such that $\frac{f(n)}{g(n)} \leq C$ for all $n \geq n_0$.  
We say $f(n) = \Omega(g(n))$ if and only if there are constants $C>0$ and $n_0 > 0$ such that $\frac{f(n)}{g(n)} \geq C$ for all $n \geq n_0$.  
Finally, we say $f(n) \sim g(n)$ if and only if $\lim_{n\to\infty} \frac{f(n)}{g(n)} = 1$. 

\section{Example 1: integer partitions}
\label{sect:integer:partitions}
\subsection{Unrestricted integer partitions}
An integer partition of size~$n$ is a collection of unordered positive integers which sum to $n$; we denote the total number of integer partitions of size~$n$ as $p(n)$. 
Let $Z_1(x), Z_2(x), \ldots$ denote a collection of \emph{independent} geometric random variables, with $\e Z_i(x) = 1-x^i$ for any $0<x<1$, $i=1,2,\ldots,n$.  
Letting $T_n \equiv T_n(x) := \sum_{i=1}^n i\, Z_i(x)$ denote the sum of the independent random variables, we have 
\begin{equation}\label{Tn} \P(T_n = n) = \sum_{z_1 + 2z_2 + \ldots n z_n = n} \P(Z_1 = z_1, \ldots, Z_n = z_n) = p(n)\ts x^n \prod_{i=1}^n (1-x^i). \end{equation}
Hence, our collection $(Z_1, \ldots, Z_n)$ satisfies Assumption~\ref{Boltzmann:assumption} with $r_I(\ell) = x^\ell \prod_{i\in I} (1-x^i)$, and, conditional on the weighted sum of the independent random variables equalling the target $n$, and interpreting $Z_i$ as the number of parts of size~$i$ in an integer partition, each of the $p(n)$ integer partitions of $n$ are equally likely to have been chosen. 
Since we can choose any $x$ between 0 and 1, an optimal choice which maximizes $\P(T_n=n)$ (see e.g., \cite{IPARCS, Fristedt, Temperley}) is $x = e^{-\pi/\sqrt{6n}}.$ 

The exact Boltzmann sampler samples from $(Z_1, Z_2, \ldots, Z_n)$ repeatedly until the event $\sum_{i=1}^n i\, Z_i = n$ is satisfied.  It is known, see~\cite{Fristedt}, that with the choice $x = e^{\pi/\sqrt{6n}}$, we have 
\[ \P(T_n = n) \sim \frac{1}{\sqrt[4]{96n^3}}, \]
and so we reject an expected $O(n^{3/4})$ samples before we obtain an integer partition of exactly size~$n$. 

Using PDC deterministic second half, an optimal choice of division is given by $\XI = (Z_2, \ldots, Z_n)$ and $X_I = (Z_1)$, with a boost factor of
\[ \mbox{boost factor} = \frac{1}{\max_k \P(Z_1 = k)} = \frac{1}{1-x} \sim \frac{\sqrt{6n}}{\pi}, \]
whence the overall total number of times we must sample from the distribution~$\L(\XI)$ is thus $O(n^{1/4})$, a noteworthy speedup; see~\cite{PDC}. 
In addition, since the geometric distribution has point probabilities which are monotonically decreasing, i.e., $\P(Z_i = k) \geq \P(Z_i = k+1)$ for all $k \geq 0$ and $i \geq 1$, the maximum point probability occurs at $k=0$, and so the acceptance event~\eqref{PDCTSH discrete} is simply 
\[ \left\{ \sum_{i=2}^n i\, Z_i \leq n \mbox{ and } U < e^{-\pi\, (n-\sum_{i=2}^n i\, Z_i)/\sqrt{6n}}\right\}, \]
where we recall that $U$ is a uniform random variable in the interval $(0,1)$.  Note that this division is optimal in choice of index $i=1$, since $\P(Z_i = 0)  = 1-x^i \geq 1-x$ for all $i = 1,2,\ldots,$ with equality for $i=1$.  

Extending the previous division, we next consider $\XI = (Z_{k+1}, Z_{k+2}, \ldots, Z_n)$ and $X_I = (Z_1, \ldots, Z_{k})$ for any $1 \leq k \leq n$. 
Then we have 
\[ \mbox{boost factor} = \frac{1}{\max_{1\leq \ell\leq n} \P(\sum_{i=1}^k i\, Z_i(x) =\ell)}.  \]
Fortunately, the extensive work in asymptotic enumeration surrounding the integer partition function and its many variations is applicable, in particular~\cite{Szekeres1, Szekeres2}, which implies that $\sum_{i=1}^k i\, Z_i$ is close to a normal distribution with maximum density asymptotically $O(\sqrt{k\, n})$ for $k = O(\sqrt{n})$. 
Thus, if we demand an expected number of rejections which is $O(n^a)$, for some $0\leq a\leq \frac{1}{4}$, then we may take any $k = \Omega\left(n^{\frac{1}{2}-2a}\right)$. 
On the other hand, if we are only willing to store a table of size $n \times O(n^b)$, for some $0 \leq b \leq \frac{1}{2}$, the expected number of rejections is then $\Omega\left(n^{\frac{1}{4} - \frac{b}{2}}\right)$ and $O\left(n^{\frac{1}{4}}\right).$ 

\begin{remark}
{\rm The case when $k = 1$ is PDC deterministic second half, whereas the case $k = \sqrt{n}$ implies a constant rejection probability, at the cost of creating an $n \times O(\sqrt{n})$ table. 
A ``middle" ground might be $k = n^{1/4}$, with a table of size $n \times O(n^{1/4})$ and an expected number of rejections $O(n^{1/8})$.  }
\end{remark}

Let us make this example even more explicit, in order to highlight its practicality. 
Recall that the number of integer partitions of $n$ into parts of size at most $k$ satisfies the recursion~\eqref{pnk:recursion}, from which we have calculated a table for values of $p(n,k)$ for  $n$ and $k$ between 1 and 10 below.  (Note: the diagonal entries are precisely $p(n)$ for $n=1,2,\ldots, 10$.)
\begin{center}
\begin{tabular}{lll}
\hline
{\tt 1 1 1 1 1\ \  1\ \ 1\ \  1\ \  1\ \  1 }\\
{\tt 1 2 2 3 3\ \  4\ \ 4\ \  5\ \  5\ \ 6 }\\
{\tt 1 2 3 4 5\ \  7\ \ 8\  10 12 14 }\\
{\tt 1 2 3 5 6\ \  9\ 11 15 18 23 }\\
{\tt 1 2 3 5 7 10 13 18 23 30 }\\
{\tt 1 2 3 5 7 11 14 20 26 35 }\\
{\tt 1 2 3 5 7 11 15 21 28 38 }\\
{\tt 1 2 3 5 7 11 15 22 29 40 }\\
{\tt 1 2 3 5 7 11 15 22 30 41 }\\
{\tt 1 2 3 5 7 11 15 22 30 42 }\\
\hline
\end{tabular}
\end{center}

We can sample a uniformly random integer partition of size~$10$ via the recursive method as follows:
looking at the final column, one generates a uniform integer between 1 and 42, say 27, which determines that the largest part is 5 since 27 lies between the values in the 4th and 5th rows.  Shifting to the 5th column, we either generate a random integer between 1 and 7, or continue to use our original value of 27 subtracted by the cutoff value of 23 in the fourth row of the tenth column.  This leaves us with the value 4, and we repeat the process in this 5th column, selecting the next largest part as 3 since $4$ lies between the value in the second and third rows.  Shifting again now to column 2, and subtracting 4 by the value in the second row of the fifth column, we obtain a 0, which means that we fill out the rest of the partition with 1s.
Thus, our partition of 10 generated in this manner is 5, 3, 1, 1. 

We now demonstrate how to apply Algorithm~\ref{PDC recursive method} to integer partitions. 
We consider the vector $(Z_1, \ldots, Z_n)$ describing an integer partition of size~$n$, and with some $k$ specified, we use the PDC division $\XI = (Z_{k+1}, Z_{k+2}, \ldots, Z_n)$ and $X_I = (Z_1, \ldots, Z_{k})$. 
The PDC algorithm is then
\begin{enumerate}
\item[(1)] Generate a table $T$ of values of $p(i,j)$ for $0 \leq i \leq k$ and $0 \leq j \leq n.$  Denote the entries in the final row by $T(j) = p(k,j)$, $j \geq 0$.  
\item[(2)] Sample from $\L(Z_{k+1}, Z_{k+2}, \ldots, Z_n)$, say observing $(z_{k+1},z_{k+2},\ldots,z_n)$, with weight $m := \sum_{i=k+1}^n i\, z_i$. 
\item[(3)] Let $\yI := n-m$.  We accept the sample with probability 
\[ t(a) = \frac{p(k, m) x^m}{\max_\ell p(k,\ell)\, x^\ell} = \frac{T(m)\, x^m}{\max_\ell T(\ell)\, x^\ell}. \]
\item[(4)] Sample from $\L\left(Z_1, \ldots, Z_{k})\, \middle|\, \sum_{i=1}^{k} i\, Z_i = \yI\right)$ from the table $T$ using the recursive method.
\end{enumerate}

For example, let us take $n=10$ and $I = \{1,2,3\}$, i.e., $k=3$.  
Rather than make a full $n \times n$ table of values, we instead only need the first three rows. 
\begin{center}
\begin{tabular}{lll}
\hline
{\tt 1 1 1 1 1\ \  1\ \ 1\ \  1\ \  1\ \  1 }\\
{\tt 1 2 2 3 3\ \  4\ \ 4\ \  5\ \  5\ \ 6 }\\
{\tt 1 2 3 4 5\ \  7\ \ 8\  10 12 14 }\\
\hline
\end{tabular}
\end{center}

The algorithm is then to sample from $(Z_4, \ldots, Z_{10})$, a vector of \emph{independent} geometric random variables, and then reject depending on the value of $\sum_{\ell=4}^{10} \ell Z_\ell$.  Let's say we observed $(z_4, \ldots, z_{10}) = (1, 0,0,0,0,0,0)$ for this first step, which corresponds to one part of size~$4$, and no parts of larger size.  
The rejection probability is then given by 
\[ t(a) = \frac{p(3,6)x^6}{\max_\ell p(3,\ell)x^{\ell}} = \frac{T(6)\, x^6}{\max_\ell T(\ell)\, x^\ell}. \]
Taking $x = e^{-\pi/\sqrt{60}}$, and multiplying each entry in the $j$th column by $x^j$, we obtain the following floating point values for the last row in the table above
\begin{center}
\begin{tabular}{ll}
Columns 1 -- 5: & 0.666591\ \ 0.888688\ \ 0.888588\ \ 0.789767\ \ 0.658065 \\
Columns 6 -- 10: & 0.614125\ \ 0.467852\ \ 0.389833\ \ 0.311831\ \ 0.242508. \\
\end{tabular}
\end{center}
The rejection probability is thus 
\[ t(a) = \frac{0.614125}{0.888688} = 0.691047. \]
Suppose we accept this sample (otherwise we would resample $(Z_4, \ldots, Z_{10})$ and apply rejection as before), then we complete the partition of size~$6$ \emph{into parts of size at most 3} by sampling from an integer between 1 and 7 and applying the recursive method starting in the 6th column. 

We end our discussion of this example with a suggestion for sampling efficiently from $\XI = (X_{k+1}, X_{k+1}, \ldots, X_n)$ for any $k \geq 0$.  
One could sample each $X_i$ via a uniform random variable $U_i$ over the unit interval $(0,1)$, and apply the transformation $\left\lfloor \frac{\ln(U_i)}{i \ln(x)}\right\rfloor$ to obtain a random variate with distribution $\L(X_i)$, $i=k+1,\ldots, n$. 
However, this requires generation of $n-k$ uniform random variables. 
It was shown in~\cite[Section~5]{PDC}, however, that the entropy in $\XI$ is $O(\sqrt{n})$ for any $k \geq 0$, and a Poisson process sampling procedure was specified which is asymptotically efficient. 
In fact, it is not difficult to show that when $k = \Omega(n^{1/2+\epsilon})$ for any $\epsilon>0$, the entropy of $\XI$ is $O(1)$, whereas the na\"ive sampling algorithm would still generate $n-k$ uniform random variables. 

\subsection{Integer partitions into distinct parts}
An example where PDC \emph{deterministic second half} is limited is the case when the random variables $Z_1, Z_2, \ldots$ are Bernoulli, which is a special case of combinatorial selections; see Section~\ref{sect:selections}. 
However, the analogous PDC with the recursive method provides a more significant improvement. 

Consider, for example, integer partitions into distinct part sizes.  I.e., we take $Z_i(x)$ to be a Bernoulli random variable with parameter $\frac{x^i}{1+x^i}$, $i=1,2,\ldots$, and any $0<x<1$.  (We could also consider equivalently the geometric random variables of the previous section conditioned to be in the set $\{0,1\}$.)  
Then, similarly as with unrestricted integer partitions, conditional on $T_n := \sum_{i=1}^n i\, Z_i(x) = n$, $Z_i$ denotes the number of parts of size~$i$ in a uniform integer partition of size~$n$ into distinct parts. 
It was shown in~\cite{Fristedt} that, taking $x = e^{-\pi/\sqrt{12\, n}}$, we optimally have 
\[ \P(T_n = n) \sim \frac{1}{\sqrt[4]{192\, n^3}}. \]

As per Remark~\ref{default:remark}, the first approach to speeding up the rejection probability is to take $\XI = (Z_2, \ldots, Z_n)$ and $X_I = (Z_1)$. 
Unfortunately, the boost factor in this setting is limited, since $\P(Z_i = 0) = \frac{x^i}{1+x^i} \geq \frac{1}{2}$ for all $i \geq 1$, with $i=1$ giving a paltry optimal boost factor of at most $2$. 
Using PDC with the recursive method, however, we obtain similar boost factors as in the unrestricted case. 

Note first we have a similar recursion.  Letting $q(n,k)$ denote the number of partitions of $n$ into distinct parts all at most $k$, we have 
\[ q(n,k) = q(n-k,k-1) + q(n,k-1), \qquad 1 \leq k \leq n, \]
with $q(k,0) = 1$ when $k \geq 0$, $q(k,n) = q(n,n)$ when $k > n$, and $q(k,n) = 0$ otherwise. 
This recursion is similar to the one for unrestricted integer partitions in~\eqref{pnk:recursion}, but since we can have at most one part of each size, if we choose to use a part of that size we must also transition from $k$ to $k-1$.  
The rest of the details are similar to the previous section, and are left as an exercise.

\section{Example 2: set partitions}
\label{sect:set:partitions}
A partition of a set $\{1,2,\ldots, n\}$ of size~$n$ is a disjoint union of sets whose union is $\{1,2,\ldots,n\}$.  The sets are called blocks, and the number of elements in a given block is called the block size. 
There is a natural mapping (surjection) from the block sizes of a set partition of size~$n$ to the part sizes of an integer partition of size~$n$. 
There is also an analogous sampling algorithm, with key differences.

For any $x>0$, let $Z_1(x), Z_2(x), \ldots$ denote a collection of independent Poisson random variables, with $\e Z_i(x) = \lambda_i = \frac{x^i}{i!}$,  for $i=1,2,\ldots,n$.  
Random variable $Z_i(x)$ counts the number of blocks of size~$i$ in a random set partition of random size, $i=1,2,\ldots$. 
The number of set partitions of size~$n$ is known as the $n$-th Bell number, often denoted by $B_n$, and satisfies the following recurrence: 
\begin{equation}\label{Bell:recursion} B_n = \sum_{i=0}^{n-1} \binom{n-1}{i} B_{i}, \qquad n \geq 2, \end{equation}
with $B_0 = B_1 = 1$. 
Let $T_n(x) = \sum_{i=1}^n i\, Z_i(x).$ 
We have for $z_1, \ldots, z_n$ satisfying $\sum_{i=1}^n i\, z_i = n$ (see e.g.,~\cite{IPARCS}),
\[ \P(Z_1 = z_1, \ldots, Z_n = z_n) =  \frac{x^n}{n!}\, \exp\left(-\sum_{i=1}^n \frac{x^i}{i!}\right), \]
whence
\[ \P(T_n(x) = n) = B_n\,  \frac{x^n}{n!}\, \exp\left(-\sum_{i=1}^n \frac{x^i}{i!}\right), \]
and so we see that Assumption~\ref{Boltzmann:assumption} is satisfied. 

It was shown in~\cite{PittelSetPartitions}, see also~\cite{MoserWymanBell}, that with $x$ satisfying $x e^x = n$, we have 
\[ \P(T_n(x) = n) \sim \frac{1}{\sqrt{2\pi n(x+1)}}. \]
Note that $x \sim \log(n)$ (see~\cite{deBruijn} for more terms in the asymptotic expansion), and so the exact Boltzmann sampling algorithm to obtain the block sizes of a uniformly generated set partition of size~$n$ has an expected $O(\sqrt{n \log n})$ number of rejections. 

It was shown in~\cite[Section 3.3.1]{PDC} that using PDC deterministic second half with $\XI = (Z_1, Z_2, \ldots, Z_{[x]-1}, Z_{[x]+1}, \ldots Z_n)$ and $X_I = (Z_{[x]})$, one obtains an optimal boost factor of 
\[ \mbox{boost factor} = \frac{1}{\max_\ell \P(Z_{[x]} = \ell)} \sim O\left(\frac{\sqrt{n}}{\log^{3/4}(n)}\right), \]
for an overall expected number of rejections of $O(\log^{5/4}(n))$. 

For this example, let us explore the recursive method on the recursion in~\eqref{Bell:recursion}. 
It was shown in~\cite[Algorithm~S]{NW} how to obtain a sampling algorithm using this recursion. 
Specifically, we first generate a new block size $n-K$ using random variable $K$ with distribution 
\[ \P(K = k) = \binom{n-1}{k} \frac{B_k}{B_n}, \qquad 0 \leq k \leq n-1, \]
which generates a given block size in its correct proportion with respect to all set partitions containing at least one block of size $n-K$.
Then we randomly sample a set of elements to place inside the block, and continue recursively with the remaining elements. 
A straightforward calculation, see~\cite{NW}, shows that a set partition generated in this way is uniform over all set partitions of size~$n$. 
We now apply Algorithm~\ref{PDC recursive method} in this setting. 

Using the heuristic from~\cite{IPARCS}, 
for some $\alpha > 0$ we choose index set 
\begin{align*}
I = \{ [x-\alpha\sqrt{x}], [x-\alpha\sqrt{x}]+1, \ldots, [x], \ldots, [x+\alpha\sqrt{x}]-1, [x+\alpha\sqrt{x}]\},
\end{align*}
where recall $x$ is the solution to $x e^x = n$, or approximately $\log(n)$. 

To sample from $\XI$, we recommend simulating a Poisson process over the interval $[0,\sum_{i \notin I} \lambda_i]$, assigning a value to $Z_\ell$ based on the number arrivals in the corresponding interval of length $\lambda_\ell$, $\ell \in \{1,\ldots,n\} \setminus I$. 
The expected number of uniform random variables in the unit interval required to run such a Poisson process to completion is given by  
\[ s(n) := \sum_{i \notin I} \lambda_i = \sum_{i=1}^{x-\alpha\sqrt{x}} \frac{x^i}{i!} + \sum_{i=x+\alpha\sqrt{x}}^{n} \frac{x^i}{i!}. \]
Let $c_\alpha:=\P(\mbox{Normal}(0,1) \geq \alpha)$ denote the tail probability of a standard normal random variable at cutoff value $\alpha$, and let Po$(x)$ denote a Poisson random variable with mean $x$.  
We have
\[ \sum_{i=1}^{x-\alpha\sqrt{x}} \frac{x^i}{i!} = e^x \left( \P(\mbox{Po}(x) \leq x-\alpha\sqrt{x})\right) - 1 \sim \frac{n}{x}\ \P\left(\mbox{Normal}(0,1) \leq - \alpha\right) \sim c_\alpha\, \frac{n}{x}; \]
\[ \sum_{x+\alpha\sqrt{x}}^n \frac{x^i}{i!} = e^x \left( \P(\mbox{Po}(x) \geq x+\alpha\sqrt{x})\right) - 1 \sim \frac{n}{x}\ \P\left(\mbox{Normal}(0,1) \geq \alpha\right) \sim c_\alpha\, \frac{n}{x}. \]
Thus, to sample $\XI$ using a Poisson process in this manner requires the generation of $s(n) = O(n/\log(n))$ uniform random variates in the unit interval.  

Next, we compute the expected value of the weighted sum over indices in $I$, viz.,
\[ \sum_{i \in I} i\, \lambda_i = x e^x \P(\mbox{Po}(x) \in [x-\alpha\sqrt{x}, x+\alpha\sqrt{x}]) \sim n(1-2c_\alpha), \]
and the standard deviation
\[ \sqrt{\sum_{i \in I} i^2\, \lambda_i} \sim \sqrt{ n(x+1)(1-2c_\alpha)}. \]
Finally, to estimate the rejection probability, we assume $\sum_{i \in I} i\, Z_i$ approximately satisfies a local central limit theorem, which yields 
\[ \frac{\max_\ell \P\left(\sum_{i\in I} i\, Z_i = \ell\right)}{\P\left(\sum_{i=1}^n i\, Z_i = n\right)} \sim \frac{1}{\sqrt{1-2c_\alpha}} =  O(1). \]

The next step of the algorithm is to make a table.  
Instead of a table generated from the recursion in~\eqref{Bell:recursion}, as was the original approach in~\cite{NW}, we consider the number of set partitions of $n$ into blocks of sizes in the set~$I$, since we shall be sampling from the random variables in $\XI$ directly first. 
That is, we require the generalization of the recursive method in~\cite[Postscript:~deux ex machina]{NWBook}, where the ``primes" are the elements in $I$.  
Let $p_I(n)$ denote the number of set partitions of $n$ into blocks of sizes in the set~$I$. 
By appealing to generating functions or recursions, see for example~\cite[Section~9.4]{IPARCS}, one obtains (for \emph{any} $I \subset \{1,2,\ldots, n\}$)
\[ p_I(n) = \sum_{i \in I} \binom{n-1}{i} p_I(i), \]
with $p_I(0) = 1$. 
Thus, the recursion above can be used to make a table which contains the quantities necessary to define the rejection probability, as well as complete the sample using the recursive method.

\section{Generalizations}
\label{sect:three:classes}

\subsection{A general probabilistic principle} 
\label{sect:iparcs}

For any $n >0$, consider an index set $I = \{i_1, i_2, \ldots \} \subset \{1,2,\ldots,n\}$. 
Let $\bfa = (a_{i_1}, a_{i_2}, \ldots),$ be a sequence of nonnegative integers, and let $\bfw = (w_{i_1}, w_{i_2}, \ldots)$ denote nonnegative real-valued weights.  
Let $N(n,\bfa, \bfw)$ denote the number of objects of weight~$n$ having $a_i$ components of size~$w_i$, $i \in I$.
Summing over all $i\in I$ gives the total weight $n = \sum_{i\in I} w_i\,a_i = {\bf w} \cdot {\bf a}$ of the object, where ${\bf w} \cdot {\bf a}$ is the usual dot product on two vectors of the same dimension.   
The examples of interest will have the following form:
\[ N(n,\bfa, \bfw) := \mathbbm{1}(\bfw\cdot \bfa = n) f(I, n) \prod_{i\in I} g_i(a_i), \]
for some functions $f$ and $g_i$, $i\in I$, with
\[ p_I(n) := \sum_{\bfa} N(n,\bfa,\bfw) \]
denoting the total number of objects of weight $n$.  

Suppose now we place a uniform distribution over the corresponding set of $p_I(n)$ combinatorial objects.  Then the number of components of size~$i$ is a random variable, say with distribution $C_i$, $i\in I$, and $C_I = (C_{i_1}, C_{i_2} \ldots)$ is the joint distribution of \emph{dependent} random component-sizes that satisfies $C_I \cdot \bfw= n$.  The distribution of $C_I$ is given by
\begin{equation}\label{C distribution}
 \P(C_I = \bfa) = \mathbbm{1}(\bfw \cdot \bfa = n) \frac{f(I,n)}{p_I(n)} \prod_{i\in I} g_i(a_i).
 \end{equation}
For each $x>0$, let independent random variables $Z_i$, $i\in I$, have distributions
\begin{equation}\label{Z distribution}
 \P(Z_i = k) = c_i(x)\, g_i(k)\, x^{w_ik}, \qquad i\in I,
 \end{equation}
where $c_i$, $i\in I$, are the normalization constants, given by
\[ c_i = \left( \sum_{k \geq 0} g_i(k) x^{w_ik} \right)^{-1}. \]
Now we can state the following theorem.

     \begin{theorem}{\cite{IPARCS}}
Assume $I \subset \{1,\ldots, n\}$.  
Let $C_I = (C_i)_{i\in I}$ denote the joint distribution of random component-sizes with distribution given by Equation~\eqref{C distribution}. 
Let $Z_I = (Z_i)_{i\in I}$ denote a vector of independent random variables with distributions given by Equation~\eqref{Z distribution}, and define $T_I = \sum_{i \in I} w_i\ts Z_i$.  Then
\[ \L(C_I) = \L(Z_I | T_I = n). \]
Furthermore, we have 
\begin{equation}\label{Tn}
 \P(T_I = n) = p_I(n) \frac{x^n}{f(I,n)} \prod_{i\in I} c_i(x). 
 \end{equation}
     \end{theorem}

Immediately, we see that Assumption~\ref{Boltzmann:assumption} is satisfied, and that the corresponding hard rejection sampling algorithm for sampling from $\L(Z_I | T_I=n)$ has an expected number of rejections which is $\P(T_I=n)^{-1}$, given in~\eqref{Tn}. 
The PDC deterministic second half improvement can be applied to any index $j \in I$, with speedup given by 
\[ \speedup = O\left( \left(\max_{k} c_j(x)\, g_j(k)\, x^{w_jk}\right)^{-1}\right), \]
and if possible one should choose $j$ such that this speedup is maximized, even though by Remark~\ref{default:remark} \emph{any} choice of $j$ will provide a speedup. 
In order to show how to apply Algorithm~\ref{PDC recursive method}, we specialize to three standard classes below.

\subsection{Selections}
\label{sect:selections}
Integer partitions of size~$n$ into distinct parts is an example of a selection: each element $\{1,2,\ldots,n\}$ is either in the partition or not in the partition.  
Selections in general allow $m_i$ different types of a component of type~$i$.  For integer partitions, this would be similar to assigning $m_i$ colors to integer $i$, and allowing at most one component of size $i$ of each color.  
We have for all $0<x<1$ and $i \in I$, 
\[ \P(Z_i = k) = \binom{m_i}{k} \left(\frac{x^i}{1+x^i}\right)^k\left(\frac{1}{1+x^i}\right)^{m_i-k}, \qquad 0 \leq k \leq m_i,\]
which is binomial.  
Letting $p_I(n)$ denote the number of such combinatorial selections of weight~$n$, we have 
\[ \P(T_I = n) = p_I(n)\ts x^n \prod_{i\in I} (1+x^i)^{m_i}. \]
The recursion given in~\cite[Equation~(158)]{IPARCS} yields 
\[ k\, p_I(k) = \sum_{i=1}^k g_I(i)  p_I(k-i), \qquad k=1,2,\ldots, \]
where 
\[ g_I(i) = - x^i\sum_{k\, |\, i} k\, m_k (-1)^{i/k}\mathbbm{1}(k \in I), \]
and 
\[ p_I(0) = 1, \] 
so that Assumption~\ref{Boltzmann:assumption} is satisfied.

\subsection{Multisets}

Unrestricted integer partitions of size~$n$ is an example of a multiset: each element $\{1,2,\ldots, n\}$ can appear any number of times in the partition.  Multisets in general allow $m_i$ different types of a component of type~$i$, similar to selections.  We have for all $0<x<1$
\[ \P(Z_i = k) = \binom{m_i+k-1}{k} (1-x^i)^{m_i}x^{i k},\qquad k=0, 1,\ldots, \]
which is negative binomial.  
Letting $p_I(n)$ denote the number of such combinatorial multisets of weight~$n$, we have 
\[ \P(T_I = n) = p_I(n)\ts x^n \prod_{i\in I} (1-x^i)^{m_i}. \]
The recursion given in~\cite[Equation~(157)]{IPARCS} yields 
\[ k\, p_I(k) = \sum_{i=1}^k g_I(i)  p_I(k-i), \qquad k=1,2,\ldots, \]
where 
\[ g_I(i) = x^i\sum_{k\, |\, i} k\, m_k \mathbbm{1}(k \in I), \]
and 
\[ p_I(0) = 1, \] 
so that Assumption~\ref{Boltzmann:assumption} is satisfied.

\subsection{Assemblies}
\label{sect:assemblies}
Assemblies are described using $Z_i$ as Poisson$(\lambda_i)$, where $\lambda_i = \frac{m_i x^i}{i!}$, $i=1,\ldots,n$, and where $m_i$ is the number of different types of a component of type $i$, $i \in I$, and $x>0$.  
Set partitions are an example of an assembly, with $m_i = 1$ for all $i \in I = \{1,2,\ldots,n\}$.
In general, we have
\begin{equation}\label{Poisson prob}
 \P(Z_1 = c_1, \ldots, Z_n = c_n) = \prod_{i=1}^n \frac{ m_i^{c_i} x^{i\, c_i}}{i!^{c_i} c_i!} e^{-\lambda_i} = x^{n} e^{-\sum_{i=1}^n \lambda_i} \prod_{i=1}^n \frac{m_i^{c_i}}{i!^{c_i} c_i!}.
 \end{equation}
Letting $p_I(n)$ denote the number of such combinatorial assemblies of weight~$n$, we have 
\[ \P(T_I = n) = p_I(n)\ts \frac{x^n}{n!} \exp\left(-\sum_{i \in I} \frac{m_i x^i}{i!}\right). \]
The recursion given in~\cite[Equation~(153)]{IPARCS} yields 
\[ k\, p_I(k) = \sum_{i=1}^k g_I(i)  p_I(k-i), \qquad k=1,2,\ldots, \]
where 
\[ g_I(i) = i \lambda_i \mathbbm{1}(i \in I), \]
and 
\[ p_I(0) = 1, \] 
so that Assumption~\ref{Boltzmann:assumption} is satisfied.

\bibliographystyle{plain}                                              
\bibliography{../../../master_bib}  

\end{document}